\documentclass[11pt,leqno]{amsart}
\usepackage{amsthm,amsfonts,amssymb,amsmath,oldgerm}
\usepackage{graphicx}
\usepackage{makecell}
\numberwithin{equation}{section}
\usepackage{fullpage}

\usepackage{hyperref}
\allowdisplaybreaks[2]

\def\l{\lambda}

\def\l{\lambda}

\newcommand\br{\begin{remark}}
\newcommand\er{\end{remark}}
\newcommand\bp{\begin{pmatrix}}
\newcommand\ep{\end{pmatrix}}
\newcommand{\be}{\begin{equation}}
\newcommand{\ee}{\end{equation}}
\newcommand\ba{\begin{equation}\begin{aligned}}
\newcommand\ea{\end{aligned}\end{equation}}

\newcommand{\bap}{\begin{app}}
\newcommand{\eap}{\end{app}}
\newcommand{\begs}{\begin{exams}}
\newcommand{\eegs}{\end{exams}}
\newcommand{\beg}{\begin{example}}
\newcommand{\eeg}{\end{example}}
\newcommand{\bpr}{\begin{proposition}}
\newcommand{\epr}{\end{proposition}}
\newcommand{\bt}{\begin{theorem}}
\newcommand{\et}{\end{theorem}}
\newcommand{\bc}{\begin{corollary}}
\newcommand{\ec}{\end{corollary}}
\newcommand{\bl}{\begin{lemma}}
\newcommand{\el}{\end{lemma}}
\newcommand{\bd}{\begin{definition}}
\newcommand{\ed}{\end{definition}}
\newcommand{\brs}{\begin{remarks}}
\newcommand{\ers}{\end{remarks}}

\newcommand{\Id}{{\rm Id }}

\newtheorem{theorem}{Theorem}[section]
\newtheorem{proposition}[theorem]{Proposition}
\newtheorem{corollary}[theorem]{Corollary}
\newtheorem{lemma}[theorem]{Lemma}

\theoremstyle{remark}
\newtheorem{remark}[theorem]{Remark}
\theoremstyle{definition}
\newtheorem{definition}[theorem]{Definition}

\newtheorem{example}[theorem]{Example}

\newcommand{\beq}{\begin{equation}}
\newcommand{\eeq}{\end{equation}}

\title{Stability of strong detonation waves for Majda's model with general ignition functions}
\author{Soyeun Jung}
\address{Kongju National University}
\email{soyjung@kongju.ac.kr}
\thanks{Research of S.J. was supported by the National Research Foundation of Korea(NRF) grant funded by
the Korea government(MSIP) (No. 2019R1F1A1063018).}

\author{Zhao Yang}
\address{Indiana University, Bloomington, IN 47405, University of Illinois at Urbana-Champaign, IL 61801}
\email{yangzha@indiana.edu, zhaouiuc@illinois.edu}
\thanks{Research of Z.Y. was supported by the College of Arts and Sciences Dissertation Year Fellowship
of Indiana University, Bloomington}
\author{Kevin Zumbrun}
\address{Indiana University, Bloomington, IN 47405}
\email{kzumbrun@indiana.edu}
\thanks{Research of K.Z. was partially supported under NSF grant no. DMS-1400555}
\begin{document}

\begin{abstract}
For strong detonation waves of the inviscid Majda model, spectral stability was established 
by Jung and Yao for waves with step-type ignition functions, by a proof based largely on 
explicit knowledge of wave profiles. 
In the present work, we extend their stability results to strong detonation waves with more
general ignition functions where explicit profiles are unknown. 
Our proof is based on reduction to a generalized Sturm-Liouville problem, 
similar to that used by Sukhtayev, Yang, and Zumbrun to study
spectral stability of hydraulic shock profiles of the Saint-Venant equations.
\end{abstract}

\date{\today}
\maketitle

\section{Introduction}\label{s:intro}
In this paper, building on methods of \cite{SYZ,SZ}, we study spectral stability of strong detonation waves \cite{BZ} of inviscid Majda's model \cite{M}
\ba \label{detonation}
u_t+\left(\frac{u^2}{2}\right)_x&=kq\phi(u)z, \\
z_t&=-k\phi(u)z. 
\ea
Here, $u\geq 0$ is a lumped variable modeling the gas-dynamical quantities of density, momentum, energy and temperature, $z \geq 0$ is the mass fraction of the reactant, $q\geq 0$ is a fixed coefficient of heat release of the reaction, and $\phi(u)$ is a ``general" ignition function satisfying 
\be \label{ignitionfunction}
\text{$\phi(u)=0$ for $u\leq u_i$, and $\phi(u)>0$ for $u>u_i$,}
\ee
where $u_i>0$ is called the ignition level. 

\smallskip

A strong detonation wave of \eqref{detonation} is a traveling wave solution of the form
\be\label{sdw}
(u, z)(x,t)=(\bar u, \bar z)(x-st), \quad \lim_{\xi \rightarrow \pm \infty} (\bar u, \bar z)(\xi)=(u_{\pm}, z_{\pm})
\ee
where $(\bar{u},\bar{z})(\xi)$ is the profile function and is smooth except at a single shock discontinuity at (without loss of generality) $\xi=0$. 
At this discontinuity, known as a ``Neumann shock", $\bar{u}$ jumps from $u_*:=\bar u(0^-)$ 
to $\bar u(0^+)$ as $\xi$ crosses zero from left to right and
the limiting states $u_{\pm}, z_{\pm}$ satisfy 
\be
z_-=0, \quad z_+=1, \quad u_+ < u_i < u_-,  \quad \text{and} \quad  u_+<s<u_-. 
\ee
At the shock $\xi=0$, the Rankine-Hugoniot condition associated with \eqref{detonation} reads
\be\label{RHC}
\left[\left(\begin{array}{c} \bar{u}^2/2-s\bar{u} \\ -s\bar{z} \end{array}\right)\right]=\left(\begin{array}{c} 0\\0\end{array}\right),
\ee 
where $[\cdot]:=\cdot|_{0^+}-\cdot|_{0^-}$ denotes jump in $\cdot$ across $\xi=0$, which yields
\be
\bar{z}(0^-)=z_+=1,\quad u_*+u_+=2s.
\ee 
See \cite{BZ, M, Z1,Z2} for further discussion.
\smallskip

It is shown in \cite{Er1,Er2,Er3,JLW,Z1,Z2} that spectral stability of detonation waves may be determined by examination of the Evans-Lopatinsky determinant $\Delta(\lambda)$ \eqref{Lopatinsky} (defined below). 
The determinant is a stability function which is analytic in the right half complex plane,
and for which absence of roots in the right half plane (save for a single ``translational'' zero
eigenvalue at the origin) is defined as spectral stability. 
Thus, the main purpose of this paper is to seek conditions needed for the general ignition functions \eqref{ignitionfunction} such that the following statement holds:
\be \label{SS condition} \tag{D}
\text{Except for a simple root at $\lambda=0$, $\Delta(\lambda)$ \eqref{Lopatinsky} has no roots in $\{ \Re \l \geq 0 \}$.}  
\ee
\smallskip

For a simple step-type ignition function $\phi(u)$ which is equal to zero for $u<u_i$ and one for $u>u_i$, 
the above condition \eqref{SS condition} has been verified in \cite{JY} by direct calculation of the 
Evans-Lopatinsky determinant $\Delta(\lambda)$. Also, in \cite{BZ}, the authors have presented a 
systematic numerical investigation of the Evans-Lopatinsky determinant with Arrehenius-type ignition 
functions. However, as far as we know, spectral stability has not been verified analytically for 
general ignition functions other than step-type. We are motivated by the recent 
approach of Sukhtayev, Yang, and Zumbrun \cite{SYZ} for investigating spectral stability of hydraulic 
shock profiles. Utilizing that framework here, we obtain the main result Theorem \ref{main}.

\section{Rescaling and construction of strong detonation waves}

We now briefly review the construction of strong detonation waves in \cite{BZ, JY}. Introducing 
the change of coordinates
\be 
\tilde{x}=\frac{x}{s},\quad \tilde{t}=t,\quad \tilde{u}=\frac{u-u_+}{s-u_+},\quad \tilde{z}=z,\quad \tilde{q}=\frac{q}{s-u_+},\quad \tilde{\phi}(\tilde{u})=k\phi(u),
\ee 
equations \eqref{detonation} become
\ba 
\label{detonation1}
\tilde{u}_{\tilde{t}}+\left(\omega\frac{\tilde{u}^2}{2}+(1-\omega)\tilde{u}\right)_{\tilde{x}}&=\tilde{q}\tilde{\phi}(\tilde{u})\tilde{z},\\ \tilde{z}_{\tilde{t}}&=-\tilde{\phi}(\tilde{u})\tilde{z}
\ea
where $\omega=\frac{s-u_+}{s}\in (0,1]$. In the new coordinates, we fix the traveling waves speed $s$ to be $1$ and $\tilde{u}_+$ to be $0$. Furthermore, we have
\be 
\tilde{z}_-=0,\quad \tilde{z}_+=1, \quad \text{and $\tilde{u}_*=\frac{u_*-u_+}{s-u_+}=2$.}
\ee
From now on, we work with \eqref{detonation1}, dropping tildes for ease of writing. 
Assume that the profile $(\bar{u},\bar{z})(\xi)$ is smooth on $\xi \gtrless 0$ with a single discontinuity at $\xi =0$. 
On the $\xi>0$ part, assume that the system holds at a quiescent (i.e. nonreacting) constant state:
\be 
(\bar{u},\bar{z})(\xi)\equiv (u_+, z_+)=(0,1), \text{ for $\xi>0$.} \ee 
At the shock $\xi=0$, our former analysis yields \be 
\bar{u}(0^-)=u_*=2,\quad \bar{z}(0^-)=1.\ee
On the $\xi<0$ part, plugging the ansatz \eqref{sdw} into \eqref{detonation1} with $s=1$, the profile ODE reads
\be 
\label{profileODE}
\omega\left(\frac{1}{2}\bar{u}^2-\bar{u}\right)'=\phi(\bar{u})q\bar{z},\quad \bar{z}'=\phi(\bar{u})\bar{z}.
\ee
Subtracting $q$ times the second equation of \eqref{profileODE} from the first equation of \eqref{profileODE} yields
\be 
\left(\frac{1}{2}\omega\bar{u}^2-\omega\bar{u}-q\bar{z}\right)'=0.
\ee 
Hence, for $\xi<0$, the quantity $\omega\bar{u}^2(\xi)/2-\omega\bar{u}(\xi)-q\bar{z}(\xi)$ is 
equal to a constant $$\frac{1}{2}\omega u_*^2-\omega u_*-q\bar{z}(0^-)=-q,
$$
yielding
\be 
\bar{u}(\xi)=1+\sqrt{1-2q(1-\bar{z}(\xi))/\omega},\quad \xi<0.
\ee 

The profile ODE \eqref{profileODE} thus reduces to the scalar ODE
\be 
\label{profileODE2}
\bar{z}'=\phi\left(1+\sqrt{1-2q(1-\bar{z})/\omega}\right)\bar{z}
\ee 
with initial condition $\bar{z}(0^-)=1$. A sufficient condition for existence of monotone increasing solution to \eqref{profileODE2} is the ignition level condition
\be 
\label{con_existprofile}
u_i<u_-.
\ee
\section{the Eigenvalue system and Evans-Lopatinsky determinant}\label{Eigenvalue system and Lopatinsky determinant}

In this section, we provide a concise derivation of the Evans-Lopatinsky determinant. For a detailed derivation, see \cite{YZ} and the references therein. Linearizing \eqref{detonation1} and its Rankine-Hugoniot condition about a detonation wave and performing Laplace transform to the linearized equations in ``good unknown" \cite{YZ,JLW,Z1,Z2}, we obtain the following eigenvalue problem
\ba  \label{EE}
&\partial_\xi(Av)=(E-\lambda\Id)v, \quad \xi \gtrless 0, &\text{interior equation,}\\
&\eta [\lambda\overline{W}-R(\overline{W})]=[Av],&\text{boundary condition,}
\ea
where $v$ is the Laplace transform of the perturbation in ``good unknown", the scalar $\eta$ is the Laplace transform of shock location, $\Id$ is an identity matrix,
\be 
A=\left[\begin{array}{cc} \omega(\bar{u}-1) & 0\\ 0 & -1 \end{array}\right],\quad
E=\left[\begin{array}{cc} q\bar{z}\phi_u(\bar{u}) & q\phi\left(\bar{u}\right)\\ -\bar{z} \phi_u \left(\bar{u}\right) & -\phi \left(\bar{u}\right) \end{array}\right],\quad \overline{W}=\left[\begin{array}{c} \bar{u}\\\bar{z} \end{array}\right],\quad R(\overline{W})=\left[\begin{array}{c} q\phi(\bar{u})\bar{z}\\-\phi(\bar{u})\bar{z}\end{array}\right],
\ee 
and $[\cdot]:=\cdot|_{0^+}-\cdot|_{0^-}$ denotes jump in $\cdot$ across $\xi=0$. With $(\bar{u},\bar{z})(\xi)$ holding at quiescent state (0,1) on $\xi>0$ part, the interior equation of \eqref{EE} readily becomes $\omega \partial_\xi v_1=\lambda v_1$, $ \partial_\xi v_2=\lambda v_2$. For $\Re\lambda\ge 0$, the trivial solution $v(\xi)=0$ is then the only $L^2$-solution on $\xi>0$ part. Therefore, we can reduce the eigenvalue problem \eqref{EE} to
\ba  \label{EE1}
&\partial_\xi(Av)=(E-\lambda\Id)v, \quad \xi < 0, &\text{interior equation,}\\
&\eta [\lambda\overline{W}-R(\overline{W})]=A(0^-)v(0^-),&\text{boundary condition.}
\ea
Furthermore, we find the limiting matrix of \eqref{EE1} 
$$A^{-1}(-\infty)\left(E(-\infty)-\lambda Id\right)=\left[\begin{array}{rr}-\lambda/\sqrt{\omega^2-2q\omega}&\phi(u_-)q/\sqrt{\omega^2-2q\omega}\\0&\lambda+\phi(u_-)\end{array}\right]$$
always has a positive real part eigenvalue and a negative real part eigenvalue for $\Re\lambda>0$. Hence, there is one decaying mode and one growing mode as $\xi\rightarrow-\infty$ of the interior equation \eqref{EE1}.

We may reformulate the boundary condition of \eqref{EE1} as the following Evans-Lopatinsky determinant.
\begin{definition}Corresponding to a strong detonation profile $\overline{W}=(\bar{u},\bar{z})^T$, we define its {\it Evans-Lopatinsky determinant} \cite{Er1,Er2,Er3,JLW,Z1,Z2} as
\be 
\label{Lopatinsky}
\Delta(\lambda)=\det\left(\left[\begin{array}{rr}[\lambda\overline{W}-R(\overline{W})] & A(0^-)v(0^-)\end{array}\right]\right)
\ee 
where $v$ is a decaying mode of the interior equation \eqref{EE1}.
\end{definition}
\begin{definition} We say a strong detonation wave is 
spectrally stable if there holds condition \eqref{SS condition}.
\end{definition}
\section{Spectral stability of strong detonation waves}\label{Spectral stability of strong detonation waves}

In this section, we prove the condition \eqref{SS condition} for ignition functions \eqref{ignitionfunction} satisfying condition \eqref{ignitioncon} below. As we mentioned in the introduction, we will perform the reduction scheme established in \cite{SYZ} for the eigenvalue problem \eqref{EE1}. 
We then extend the spectral stability result for step-type ignition functions in \cite{JY} to the case 
of ignition functions satisfying \eqref{ignitionfunction}, using a homotopy argument. 
We begin with the following lemma to show that $\l=0$ is a simple root of $\eqref{Lopatinsky}$. 

\begin{lemma}
\label{simpleroot}
$\lambda=0$ is a simple root of the Evans-Lopatinsky determinant \eqref{Lopatinsky} 
	if and only if $\phi(u_*)=\phi(2)\neq 0$, in particular under assumptions
	\eqref{ignitionfunction} and \eqref{con_existprofile}.
\end{lemma}
\begin{proof}
Setting $\lambda=0$, the interior equation becomes $(Av)'=Ev$. The eigenvalues of $E(-\infty)A^{-1}(-\infty)$ are $0$ and $\phi(u_-)>0$. Therefore, the decaying manifold as $\xi \rightarrow-\infty$ is one dimensional. 
Taking without loss of generality $v=\overline{W}'$, we thus have
\be 
\Delta(0)=\det\left(\left[\begin{array}{rr}R(\overline{W}(0^-)) & A(0^-)\overline{W}'(0^-)\end{array}\right]\right)=\det\left(\left[\begin{array}{rr}R(\overline{W}(0^-)) & R(\overline{W}(0^-)) \end{array}\right]\right)=0. 
\ee 
To check simplicity of the root, it suffices to show $\Delta_\lambda(0)\neq 0$. 
Differentiating \eqref{Lopatinsky} and setting $\lambda=0$ yields
\ba 
\Delta_\lambda(0)&=\det\left(\left[\begin{array}{rr}[\overline{W}] & R(0^-)\end{array}\right]\right)+\det\left(\left[\begin{array}{rr}[-R(\overline{W})] & A(0^-)v_\lambda(0^-)\end{array}\right]\right)\\
&=\det\left(\left[\begin{array}{rr}-2 & q\phi(2)\\ 0& -\phi(2)\end{array}\right]\right)+\det\left(\left[\begin{array}{rr}q\phi(2) & \big(A(0^-)v_\lambda(0^-)\big)_1\\ -\phi(2) & \big(A(0^-)v_\lambda(0^-)\big)_2\end{array}\right]\right)\\
&=2\phi(2)+\left(q\big(A(0^-)v_\lambda(0^-)\big)_2+\big(A(0^-)v_\lambda(0^-)\big)_1\right)\phi(2).
\ea 
A simple calculation shows that $\tilde{v}(\xi):=q\big(A(\xi)v_\lambda(\xi)\big)_2+\big(A(\xi)v_\lambda(\xi)\big)_1$ satisfies 
\be 
\tilde{v}'=(-q\bar{z}-\bar{u})',\quad \text{for $\xi<0$}.
\ee 
Integrating from $-\infty$ to $\xi<0$ yields \be \tilde{v}(\xi)=(-q\bar{z}-\bar{u})(\xi)-(-q\bar{z}-\bar{u})(-\infty).\ee  Therefore, $\Delta_\lambda(0)=2\phi(2)+\big(-q-1+\sqrt{1-2q/\omega}\big)\phi(2)=\big(1-q+\sqrt{1-2q/\omega}\big)\phi(2)\neq 0$ provided that $\phi(2)\neq 0$. 
However, the condition $\phi(2)\neq 0$ is negligible under \eqref{ignitionfunction}, 
\eqref{con_existprofile}, since then $\phi(u)>0$ for $u>u_i$ and $u_i<u_-<2$. 
\end{proof}

\smallskip

We now prove that $\Delta(\l) \neq 0$ for a pure imaginary eigenvalue $\l$. Following the reduction scheme in \cite{SYZ} section 2.2 and choosing
$$
T_1=\left[\begin{array}{cc} 1 & q\\ -\frac{1}{q} & 0 \end{array}\right],\quad T_2=\left[\begin{array}{cc} 1 & 0\\ -\frac{\omega(1-\bar{u})}{q} & 1 \end{array}\right],
$$
the new variable $u:=T_2^{-1}v$ satisfies
\be 
\label{ueq}
\left[\begin{array}{cc} 0 & -q\\ \frac{\omega(1-\bar{u})}{q} & 0 \end{array}\right]\left[\begin{array}{c} u_1\\ u_2 \end{array}\right]'=\left[\begin{array}{cc} \lambda (\omega(1-\bar{u})-1)& -\lambda q\\ \frac{\lambda+\omega\bar{u}_x+\omega\phi(\bar{u}) -\omega\bar{u}\phi(\bar{u}) -\phi_u(\bar{u})q\bar{z}}{q} & -\phi(\bar{u})  \end{array}\right]\left[\begin{array}{c} u_1\\ u_2 \end{array}\right].
\ee 
Solving for $u_1$ by the first equation of \eqref{ueq} and plugging it in the second equation of \eqref{ueq} yields a second order scalar ODE
\be
u_2''+(f_1\lambda +f_2) u_2'+(f_3\lambda^2+f_4\lambda)u_2=0
\ee
where 
\ba 
f_1&=-\frac{\omega\bar{u}-1-\omega}{\omega(\bar{u}-1)} ,\quad\quad f_3=-\frac{1}{\omega(\bar{u}-1)},\\
f_2&=-\frac{\omega ^2\phi(\bar{u}) -\omega \bar{u}_x-\omega \phi(\bar{u}) -2\omega ^2\phi(\bar{u}) \bar{u}+\omega ^2\phi(\bar{u}) \bar{u}^2+\phi_u(\bar{u})q\bar{z}+\omega \phi(\bar{u}) \bar{u}-\phi_u(\bar{u})\omega q\bar{z}+\phi_u(\bar{u})\omega q\bar{u}\bar{z}}{\omega \left(\bar{u}-1\right)\left(\omega \bar{u}-\omega +1\right)},\\
f_4&=-\frac{\phi(\bar{u}) -\omega \phi(\bar{u}) +\omega \bar{u}_x-\phi_u(\bar{u})q\bar{z}+\omega \phi(\bar{u}) \bar{u}+\phi_u(\bar{u})\omega q\bar{z}-\phi_u(\bar{u})\omega q\bar{u}\bar{z}}{\omega \left(\bar{u}-1\right)\left(\omega \bar{u}-\omega +1\right)}.
\ea

After a Liouville-type transformation, we have 
$$
w(\lambda,\xi)=e^{\frac{1}{2}\int_0^{\xi} f_1(y)\lambda+f_2(y)dy}u_2(\lambda, \xi)
$$
which gives
\be 
\label{weq}
w''+\left(\left(f_3-\frac{1}{4}f_1^2\right)\lambda^2+\left(f_4-\frac{1}{2}f_1f_2-\frac{1}{2}f_1'\right)\lambda-\frac{1}{4}f_2^2-\frac{1}{2}f_2'\right)w=0. 
\ee 
Noting, as in \cite{SYZ}, that the limiting constant-coefficient equation associated with \eqref{weq} as
$\xi\to -\infty$ has eigenvalues that are negatives of each other, yet at the same time are constant real
shifts of the eigenvalues associated with the limiting version of the original system in $u$ coordinates,
which are known to have real parts of different signs for $\Re \lambda \geq 0$, we readily find that
on $\Re\lambda\geq 0$, bounded solutions of \eqref{weq} are in one-to-one correspondence with bounded
solutions of the original system, and exponentially decaying in $w$ coordinates.
This confirms that zeros of the Evans-Lopatinsky determinant for the original system correspond to
exponentially decaying eigenfunctions of \eqref{weq}, which we now investigate.

In $w$ coordinate, after substituting $$\bar{u}(0^+)=0,\quad \bar{z}(0^+)=1,\quad \phi(0)=0,\quad \bar{u}(0^-)=2,\quad \bar{z}(0^-)=1,\quad \bar{u}_{\xi}(0^-)=\phi(2)q/\omega,$$ the Evans-Lopatinsky condition \eqref{Lopatinsky} $\delta(\lambda)=0$ gives boundary condition
\be 
\label{bdcondition}
w'(0^-)=-\left(\lambda\frac{\omega+1}{2\omega} +\frac{\phi_u(2)q+\omega \phi(2) -2\phi(2) q+\omega ^2\phi(2) -\omega ^2\phi(2) q+\phi_u(2)\omega q-2\omega \phi(2) q}{{2\left(\omega +1\right)\omega}}\right)w(0^-)
\ee 
Taking the $L^2$ inner product of $w$ with \eqref{weq} on the half line $\xi< 0$ yields
\ba 
\label{L2equation}
&\bar{w}(0)\cdot w'(0)-\langle w',w'\rangle\\ &+\Big\langle w,\left(\left(f_3-\frac{1}{4}f_1^2\right)\lambda^2+\left(f_4-\frac{1}{2}f_1f_2-\frac{1}{2}f_1'\right)\lambda-\frac{1}{4}f_2^2-\frac{1}{2}f_2'\right)w\Big\rangle=0.
\ea  
Equations \eqref{bdcondition} and \eqref{L2equation} yield the following lemma. 

\begin{lemma}
\label{nopureim}
The system \eqref{L2equation} has no nonzero pure imaginary eigenvalue for ignition functions satisfying
\be 
\label{ignitioncon}
\frac{d}{du}\ln(\phi(u))\le\frac{2\omega  \left(u-1\right)}{\omega u^2-2\omega u+2q},
\ee
for $1+\sqrt{1-2q/\omega}<u\le 2$. 
\end{lemma}

\begin{proof}
Substituting $\lambda=ia$, $a \neq 0$, and \eqref{bdcondition} into equation \eqref{L2equation}, 
and taking the imaginary part gives
\be \label{impart}
a\left(-\bar{w}(0)\cdot w(0)\frac{\omega+1}{2\omega}+\Big\langle w,\left(f_4-\frac{1}{2}f_1f_2-\frac{1}{2}f_1'\right)w\Big\rangle\right)=0,
\ee 
which will only have the trivial solution (ruling out that $\lambda=ia$ is an eigenvalue) provided that 
\be 
\label{cc1}
\Big(f_4-\frac{1}{2}f_1f_2-\frac{1}{2}f_1'\Big)(\xi) \leq 0,\quad  \text{for $\xi<0$.}
\ee 
	For, then, \eqref{impart} gives $w(0)=$, hence $w'(0)=0$, and so $w\equiv 0$ by solution of the
	Cauchy problem for the second-order interior equation.

We readily find
\be 
f_4-\frac{1}{2}f_1f_2-\frac{1}{2}f_1'=\frac{\left(\omega \phi(\bar{u}) +\phi_u(\bar{u})q\bar{z}-\omega \phi(\bar{u}) \bar{u}\right)\left(\omega \bar{u}-\omega +1\right)}{2\omega ^2{\left(\bar{u}-1\right)}^2}.
\ee 
Substituting $\bar{z}=(\omega\bar{u}^2 - 2\omega\bar{u} + 2q)/(2q)$ yields
\be 
f_4-\frac{1}{2}f_1f_2-\frac{1}{2}f_1'=\frac{\left(\omega \bar{u}-\omega +1\right)\left(\omega \bar{u}^2-2\omega \bar{u}+2q\right)}{4\omega ^2{\left(\bar{u}-1\right)}^2}\phi_u(\bar{u})-\frac{ \omega \bar{u}-\omega +1}{2\omega \left(\bar{u}-1\right)}\phi(\bar{u}). 
\ee 
Here, it is easy to see $\omega\bar{u}^2-2\omega\bar{u}+2q>\omega\bar{u}(-\infty)^2-2\omega\bar{u}(-\infty)+2q=0$. Hence, the condition \eqref{cc1} is equivalent to \eqref{ignitioncon}.
\end{proof}
\begin{remark}
The condition \eqref{ignitioncon} says that the rate of change of logarithm of the 
ignition function cannot be big. Moreover,  we find that $\frac{2\omega(u-1)}{\omega u^2-2\omega u+2\,q}$ is decreasing on $u\in( 1+\sqrt{1-2q/\omega},2]$; hence a sufficient condition for \eqref{ignitioncon} is 
\be 
\label{ignitioncon_sub}
\ln(\phi(u))'<\omega/q.
\ee 
\end{remark}

\smallskip

We are now ready to prove the main theorem of this paper by a homotopy argument. 
\begin{theorem}
\label{main}
The strong detonation waves of \eqref{detonation} corresponding to ignition functions satisfying 
\be 
\label{ignitionconori}
\frac{d}{du}\ln(\phi(u))\le \frac{2u-u_*-u_+}{(u-u_*)(u-u_+)+q(u_*+u_+)},\quad \text{for all $u\in (u_-,u_*]$}
\ee 
are all weakly spectral stable.
\end{theorem}
\begin{proof}
It has been verified in \cite{JY} that strong denotation waves of \eqref{detonation1} with step ignition function $\phi_0(u)$ are spectrally stable.\footnote{In version 2 of the paper, the authors corrected a minor issue in their paper published in Quarterly of Applied Mathematics. They now allow $u_+ \geq 0$
(not only $u_+=0$ in first version). This fix allows us to get spectral stability of strong detonation of equation \eqref{detonation1} with $\omega\in(0,1]$ and step ignition function.}
	Let $\phi(u)$ be an ignition function satisfying \eqref{ignitioncon} and define $\phi(r,u)=\phi^r(u)\phi_0^{1-r}(u)$, for $0\leq r\leq 1$. We have
\be 
\frac{d}{du}\ln(\phi(r,u))=r\frac{d}{du}\ln(\phi(u))+(1-r)\frac{d}{du}\ln(\phi_0(u))=r\frac{d}{du}\ln(\phi(u))\le\frac{2\omega(u-1)}{\omega u^2-2\omega u+2\,q},
\ee 
for $1+\sqrt{1-2q/\omega}<u\le 2$. That is, the family of function $\phi(r,u)$ parameterized by $r$ always satisfies \eqref{ignitioncon}. 
Hence when varying $r$ from $1$ to $0$, the unstable/stable eigenvalues cannot cross the imaginary axis by Lemma \ref{simpleroot} and Lemma \ref{nopureim}. Because there is no unstable eigenvalue for $r=0$, there must be no unstable eigenvalue for $r=1$ also. Writing condition \eqref{ignitioncon} back in original coordinates, we get condition \eqref{ignitionconori}.

\end{proof}
\begin{remark}
	For Arrhenius type ignition functions \cite{LZ} 
\be
\phi(u)=
\begin{cases}
Ce^{-\mathcal{E}/T(u)} & T>0, \\
0 & T \leq 0,
\end{cases}
\ee
investigated in \cite{BZ}, the condition \eqref{ignitioncon} becomes
\be 
\label{Tcondi}
\frac{\mathcal{E}T_u(u)}{T^2(u)}\le \frac{2\omega(u-1)}{\omega u^2-2\omega u+2q}=\frac{2(u-1)}{ u^2-2 u+2q/\omega},\quad \text{for $1+\sqrt{1-2q/\omega}<u\le 2$.}
\ee 
Specifying to the first choices of $T(u)$ in the numerical investigation in \cite{BZ}
$$
T_1(u)=1-(u-1.5)^2,
$$
our criterion \eqref{Tcondi} gives a curve on the $(q/\omega)-\mathcal{E}$ plane and validates spectral stability of points to the left of the curve. See figure \ref{fig1} (a). We also plot the points $\{q/\omega,\mathcal{E}\} = \{0.01 : 0.01 : 0.49\}\times\{0 : 0.1 :5, 5.2 : 0.2 : 10, 12, 15, 20, 30, 40\}$ studied in \cite{BZ} on figure \ref{fig1} (a). We see that most ($3963$ out of $3969$) of the points 
	studied by Barker and Zumbrun can be validated by criterion \eqref{Tcondi} as being spectrally stable. 
	There are six points $\{q/\omega, \mathcal{E}\} = \{0.49\}\times\{20\}, \{0.48, 0.49\}\times\{30\}, \{0.47, 0.48, 0.49\}\times\{40\}$ to the right of the curve, for which stability is not
	determined by \eqref{Tcondi}.
	The latter were among points
	for which Barker and Zumbrun reported numerical difficulties; however, redoing the computations with
	Matlab's stiff ODE solver ode15s appears to resolve these difficulties, yielding numerically observed
	stability.

Specifying to the second choices of $T(u)$ in the numerical investigation in \cite{BZ}
$$
T_2(u)=u,
$$
our criterion \eqref{Tcondi} gives a curve $\mathcal{E}=4\omega/q$ on the $(q/\omega)-\mathcal{E}$ plane and validates spectral stability of points to the left of the curve. See figure \ref{fig1} (b). We also plot the points $\{q/\omega,\mathcal{E}\} = \{0.01 : 0.01 : 0.37, 0.375, 0.38 : 0.01 : 0.49\}\times\{0 : 0.1 : 5, 5.2 : 0.2 : 10, 12, 15\} \cup \{0.01 : 0.01 : 0.37, 0.375, 0.38 : 0.01 : 0.47\}\times\{20\} \cup\{0.01 : 0.01 : 0.37, 0.375, 0.38 : 0.01 : 0.45\}\times\{25\} \cup  \{0.01 : 0.01 : 0.37, 0.375, 0.38 : 0.01 : 0.40\}\times\{30\} $ studied in \cite{BZ} on figure \ref{fig1} (b). We see that most ($3851$ out of $4035$) 
of the points studied by Barker and Zumbrun can be validated by criterion \eqref{Tcondi} 
as being spectrally stable. 
\end{remark}
\begin{figure}
    \centering
    \includegraphics[scale=0.3]{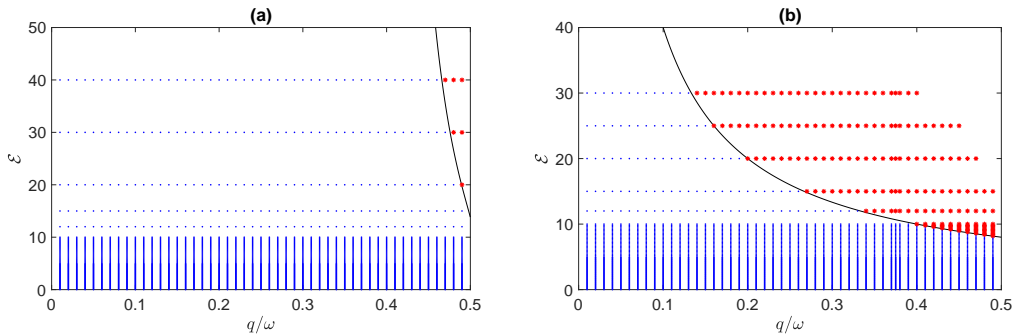}
    \caption{(a) The first choice of ignition function $T_1(u)=1-(u-1.5)^2$ in \cite{BZ}. (b) The second choice of ignition function $T_2(u)=u$ in \cite{BZ}. }
    \label{fig1}
\end{figure}

\section{Discussion and open problems}\label{discussion}
In the analyses of both \cite{SYZ} and the more general \cite{SZ}, a strict version of sign condition
\eqref{cc1} is assumed from the begining. Thus, the equivalent condition \eqref{ignitionconori}
obtained here is the strongest criterion that can be obtained by the methods of those papers.
However, evidently, this condition is not sharp. For, it is a closed condition, whereas the condition
of spectral stability is an open one, by continuity of spectra under perturbations in wave parameters.
Thus, waves close enough to a wave satisfying \eqref{ignitionconori} are stable even though they may not
satisfy \eqref{ignitionconori} themselves.
This perhaps sheds light on the extent to which one can push Sturm-Liouville methods in this context.
It would be very interesting of course to find alternative methods counting eigenvalues crossing the 
imaginary axis as well as the origin, generalizing \cite{SZ} and extending our results here to
more general choices of ignition function.

\end{document}